\documentclass[leqno,12pt]{amsart}
\usepackage{latexsym}
\usepackage{mathrsfs}
\usepackage{amsmath}
\usepackage{amssymb}
\usepackage[bookmarks]{hyperref}

\usepackage[dvipsnames]{xcolor}

\def\Jcole#1{{#1_\omega}}
\def\jt{\Jcole}

\def\JCantor{\mathcal{C}}

\newtheorem{theorem}{Theorem}[section]
\newtheorem{lemma}[theorem]{Lemma}
\newtheorem{corollary}[theorem]{Corollary}
\newtheorem{proposition}[theorem]{Proposition}

\theoremstyle{definition}

\newtheorem{definition}[theorem]{Definition}

\newtheorem{question}[theorem]{Question}

\newcommand{\C}{\mathbb{C}}

\newcommand{\N}{\mathbb{N}}

\newcommand{\Z}{\mathbb{Z}}

\newcommand{\R}{\mathbb{R}}
\newcommand{\T}{\mathbb{T}}

\newcommand{\sB}{{\scr B}}
\newcommand{\sC}{{\scr C}}

\newcommand{\sF}{{\scr F}}
\newcommand{\sG}{{\scr G}}
\newcommand{\sH}{{\scr H}}

\newcommand{\sJ}{{\scr J}}
\newcommand{\sK}{{\scr K}}

\newcommand{\bthm}{\begin{theorem}}
\newcommand{\ethm}{\end{theorem}}
\newcommand{\blem}{\begin{lemma}}
\newcommand{\elem}{\end{lemma}}
\newcommand{\bcor}{\begin{corollary}}
\newcommand{\ecor}{\end{corollary}}
\newcommand{\bprop}{\begin{proposition}}
\newcommand{\eprop}{\end{proposition}}
\newcommand{\bdefn}{\begin{definition}}
\newcommand{\edefn}{\end{definition}}
\newcommand{\bpf}{\begin{proof}}
\newcommand{\epf}{\end{proof}}

\font\teneufm=eufm10 scaled \magstep1
   \font\seveneufm=eufm7 scaled \magstep1
   \font\fiveeufm=eufm5 scaled \magstep1
  \newfam\eufmfam
  \textfont\eufmfam=\teneufm
   \scriptfont\eufmfam=\seveneufm
   \scriptscriptfont\eufmfam=\fiveeufm

\font\tenscr=rsfs10  scaled \magstep1

\font\sevenscr=rsfs7  scaled \magstep1
\font\fivescr=rsfs5  scaled \magstep1
\skewchar\tenscr='177 \skewchar\sevenscr='177 \skewchar\fivescr='177
\newfam\scrfam \textfont\scrfam=\tenscr \scriptfont\scrfam=\sevenscr
\scriptscriptfont\scrfam=\fivescr
\def\scr{\fam\scrfam}

\def\vep {\varepsilon}
\def \sm {\setminus}

\def\tsigma {\widetilde\sigma}
\def\ta {\widetilde a}

\def\tB {\widetilde B}

\def\tss{\widetilde {\sigma^*}}
\def\tfs{\widetilde {f^*}}

\def\talpha{\widetilde \alpha}

\def \ma {\mathfrak{M}_A}

\def\crx#1{C_\R(#1)}

\newcommand{\ra}{\rightarrow}

\newcommand{\ol}{\overline}

\def\pair #1{(#1_1, #1_2)}

\def\xaj#1{{{X_{A_#1}}}}
\def\inv{\mathop{\rm inv}\nolimits}

\begin{document}
\title[regular on the Cantor set]{A nontrivial uniform algebra\\ regular on the Cantor set}
\author{J. F. Feinstein}
\address{School of Mathematical Sciences, University of Nottingham, University Park, Nottingham NG7 2RD, UK }
\email{Joel.Feinstein@nottingham.ac.uk}
\author{Alexander J. Izzo}
\address{Department of Mathematics and Statistics, Bowling Green State University, Bowling Green, OH 43403}
\email{aizzo@bgsu.edu}
\thanks{This research was supported through the program \lq\lq Oberwolfach Research Fellows\rq\rq\ by the Mathematisches Forschungsinstitut Oberwolfach in 2024.   The second author was also partially supported by NSF Grant DMS-1856010.}

\subjclass[2020]{Primary 46J10, 46J15, 30H50, Secondary 30H05, 30H10}
\keywords {uniform algebra, regular uniform algebra, normal uniform algebra, logmodular, essential, Shilov boundary, Cantor set, root extension}

\begin{abstract}
We prove$\vphantom{\widehat{\widehat{\widehat{\widehat{\widehat{\widehat{\widehat X}}}}}}}$ the existence of a nontrivial uniform algebra that is logmodular and regular on the Cantor set.  As a consequence, we obtain that for every compact metrizable space $X$ without isolated points there exists a nontrivial essential uniform algebra that is logmodular and regular on $X$.  In particular, there exists a nontrivial essential uniform algebra that is logmodular and regular on the closed unit interval.  Our algebras seem to be the first known uniform algebras that are regular on a metrizable space but are not normal.
\end{abstract}

\maketitle

\vskip -2.21  true in
\centerline{\footnotesize\it Dedicated to Pamela Gorkin}
\vskip 2.21 true in

%
%
%
%

\section{Introduction}

The motivation for this paper comes from a long standing question of I. M. Gelfand~\cite{Gelfand}: Is the algebra $C([0,1])$ the only uniform algebra with maximal ideal space the closed unit interval $[0,1]$?  One might be tempted to imagine that $C([0,1])$ is the only uniform algebra defined on $[0,1]$, but many counterexamples have been given, the first of which was found by John Wermer~\cite{Wermer-Cantor}.  Toward answering Gelfand's question, it is natural to consider various uniform algebra properties and determine which of them forces a uniform algebra defined on $[0,1]$ to be $C([0,1])$.  Donald Wilken~\cite{Wilken} proved that every strongly regular uniform algebra on $[0,1]$ is $C([0,1])$.  It is unknown whether every normal uniform algebra on $[0,1]$ is $C([0,1])$.  We will answer the related question of whether every uniform algebra that is regular on $[0,1]$ is $C([0,1])$ by constructing a counterexample.  Actually, a more fundamental question is whether every uniform algebra that is regular on the Cantor set is trivial, and the bulk of the paper will be devoted to constructing an example showing that this is not the case.
As we will see, once this is done it follows easily that for every compact Hausdorff space $X$ that contains a Cantor set, there exists a nontrivial uniform algebra that is regular on $X$.  In particular, there exists a nontrivial uniform algebra that is regular on the closed unit interval $[0,1]$.

Throughout the paper, all spaces will be assumed to be Hausdorff. We will recall in the next section some of the standard terminology and notation used in this introduction and throughout the paper.

It is well known that every uniform algebra that is regular on its maximal ideal space is normal. The first nontrivial normal uniform algebra was found by Robert McKissick \cite{Mc}. However, the  first example of a nontrivial uniform algebra that is regular on its Shilov boundary appears to be the fiber algebra $A_\alpha$ 
of  Kenneth Hoffman and Isadore Singer \cite[pp.~187--190]{Hoffman} (see also \cite[Section 7]{Hoffman-Singer}). The fiber algebra $A_\alpha$ seems also to be the only example in the literature (until now) of a uniform algebra that is regular on its Shilov boundary but is not normal.  (Note that a uniform algebra that is regular on its Shilov boundary is normal if and only if its Shilov boundary coincides with its maximal ideal space.) In fact, $A_\alpha$ has the stronger property that $|A_\alpha|$ ($=\{ |f| : f\in A_\alpha\}$) is normal on the Shilov boundary for $A_\alpha$. However, the Shilov boundary for $A_\alpha$ is not metrizable. Our examples appear to be the first known examples of uniform algebras $A$ on compact \emph{metrizable} spaces $X$ such that $A$ is regular on $X$ but $A$ is not normal.  Like $A_\alpha$, our examples have the property that $|A|$ is normal on $X$.

The fiber algebra $A_\alpha$  has the additional property that it is logmodular, and our approach allows us to arrange that our examples are logmodular as well.
In fact, it seems that we have constructed the first known example of a nontrivial logmodular algebra on the Cantor set. 
There is a nontrivial Dirichlet algebra on $[0,1]$ due to Andrew Browder and Wermer~\cite{Browder-Wermer}.  This suggests the following question which remains open.

\begin{question}
Does there exist a nontrivial uniform algebra $A$ on $[0,1]$ such that $A$ is Dirichlet and regular on $[0,1]$?
\end{question}

Since every Dirichlet algebra on a totally disconnected compact space is trivial~\cite[p.~182]{Hoffman}, the approach used in this paper will not yield such an algebra.
It seems that the only construction of a nontrivial Dirichlet algebra known to be regular on a space is the one recently given by the second author in~\cite[Theorem~1.7]{Izzo}.  The Dirichlet algebra given there is in fact strongly regular and is defined on a compact metrizable space of topological dimension 1 with trivial \v Cech cohomology in all degrees.  As already mentioned, by a theorem of Wilken there are no nontrivial strongly regular uniform algebras on $[0,1]$.

We now present our main results and discuss our approach to proving them.
Our main theorem is the following.

\bthm\label{main-theorem}
There exists a nontrivial uniform algebra $A$ on the Cantor set $\JCantor$ such that $A$ is logmodular and regular on $\JCantor$; furthermore $A$ can be chosen such that $|A|$ is normal on $\JCantor$.
\ethm

In connection with this theorem, note that no nontrivial uniform algebra on a totally disconnected space can be normal.  This  follows immediately from the definition of normality for a uniform algebra and the Stone-Weierstrass theorem since for every pair of distinct points $p$ and $q$ in a totally disconnected compact space there is a separation $P, Q$ of the space with $p\in P$ and $q\in Q$.

Given a compact space $X$ and a subset $B$ of $C(X)$, the quotient space of $X$ whose elements are the common level sets of the functions in $B$ is a compact Hausdorff space which we will denote by $X_B$.
When the set $B$ is a uniformly closed unital subalgebra of $C(X)$, the functions on $X_B$ induced by the functions in $B$ form a uniform algebra on $X_B$ which we will denote by $\tB$.  Explicitly, $\tB=\{ f\in C(X_B) : f\circ \pi \in B\}$, where $\pi: X\ra X_B$ is the canonical quotient map.

We will establish Theorem~\ref{main-theorem} by applying the following result to the Hoffman-Singer fiber algebra.

\bthm\label{main-construction}
Let $A$ be a nontrivial uniform algebra on a compact totally disconnected space $X$.  If $|A|$ is normal on $X$, or if $A$ is logmodular on $X$, or if $|A|$ is normal on $X$ and $A$ is logmodular on $X$, then there exists a separable closed unital subalgebra $B$ of $A$ such that
\begin{enumerate}
\item[(a)] $|\tB|$ is normal on $X_B$, or $\tB$ is logmodular on $X_B$, or $|\tB|$ is normal on $X_B$ and $\tB$ is logmodular on $X_B$, respectively,
\item[(b)] $\tB$ is a nontrivial uniform algebra, and
\item[(c)] $X_B$ is totally disconnected.
\end{enumerate}
\ethm

From Theorem~\ref{main-theorem} we will obtain results for more general compact spaces, provided that they contain a Cantor set.  
The uniform algebras in each of these theorems are regular on the given space but are not normal.  We remind the reader that there are no nontrivial normal uniform algebras on a totally disconnected space.  As noted earlier, the existence of nontrivial normal Dirichlet algebras is proven in \cite{Izzo}.

\bthm\label{examples-arbitrary-spaces}
Let $X$ be a compact space that contains a Cantor set.  Then there exists a nontrivial uniform algebra $A$ on $X$ such that $A$ is logmodular and $|A|$ is normal on $X$ but $A$ is not normal.
\ethm

As an immediate corollary of Theorem \ref{examples-arbitrary-spaces} we obtain a nontrivial uniform algebra on $[0,1]$ that is regular on $[0,1]$, and indeed has the stronger properties provided by that theorem, thus achieving the objective presented in the opening paragraph of our paper.

\bcor\label{unit-interval}
There exists a nontrivial uniform algebra $A$ on $[0,1]$ such that $A$ is logmodular and $|A|$ is normal on $[0,1]$ but $A$ is not normal.
\ecor

Furthermore, using an approach we introduced in \cite{FI}, we will obtain \emph{essential} uniform algebras of this type on every compact metrizable space without isolated points.

\bthm\label{essential-arbitrary-spaces}
Let $X$ be a compact metrizable space without isolated points.  Then there exists a nontrivial essential uniform algebra $A$ on $X$ such that $A$ is logmodular and $|A|$ is normal on $X$ but $A$ is not normal.
\ethm

As the above consequences of Theorem~\ref{main-theorem} demonstrate, the Cantor set plays a key role in studying the existence of uniform algebras with particular properties.  It is thus of interest to investigate further what properties uniform algebras on the Cantor set can possess.  One property of particular interest is the presence of dense squares.  The presence of dense squares in a uniform algebra $A$ implies that $A$ has no nontrivial Gleason parts, no nonzero bounded point derivations, and in particular, no analytic discs in its maximal ideal space.  
A theorem of E.~M.~Chirka asserts that there are no nontrivial uniform algebras with dense squares on locally connected spaces \cite[Theorem~13.15 and Lemma~13.16]{Stout} (see also \cite{Chirka}).  In contrast, a construction of Brian Cole yields uniform algebras that do have dense squares (or even every element a square) on certain nonlocally connected spaces.  In general though, as remarked by Lee Stout \cite[p.~200]{Stout} \lq\lq it is not at all evident what topological properties" these spaces have.  However, it can be seen from the arguments in the last section of our paper that applying a certain form of the Cole construction to any uniform algebra on the Cantor set yields another uniform algebra on the Cantor set that has dense squares.  By applying the construction to the uniform algebra of Theorem~\ref{main-theorem} we will obtain the following result.

\bthm\label{optional-theorem}
There exists a nontrivial uniform algebra $A$ on the Cantor set $\JCantor$ such that the set $\{f^2:f\in A\}$ is dense in $A$ and $|A|$ is normal on $\JCantor$.
\ethm

We do not know what can be said about logmodularity of uniform algebras with the properties in this theorem.

In Section~\ref{prelim} we recall some standard notation, terminology, and results.
In Section~\ref{further-prelim} we present results that ensure that certain spaces that will arise are free of isolated points.
In Section~\ref{section-main-construction} we prove Theorems~\ref{main-theorem} and \ref{main-construction}.  In Section~\ref{general-spaces} we prove Theorem~\ref{examples-arbitrary-spaces}, Corollary~\ref{unit-interval}, and Theorem~\ref{essential-arbitrary-spaces}, along with some related results concerning preservation 
of certain uniform algebra properties by the construction introduced in~\cite{FI}.  In Section~\ref{dense-squares} we prove that Theorem~\ref{optional-theorem} follows from Theorem~\ref{main-theorem} by using Cole's method of root extensions.

It is a pleasure to dedicate this paper to Pamela Gorkin, a major contributor to the study of $H^\infty$ with whom the authors have always enjoyed talking.


\section{Preliminaries}\label{prelim}

In this section we recall some standard notation and terminology, and 
some of the standard theory of uniform algebras, as can be found in (for example) \cite{Stout}. Throughout the section, $X$ denotes a compact (Hausdorff) space.
Also, except where otherwise specified, we work with complex algebras.

We denote the set of strictly positive integers by $\N$.

A \emph{clopen} subset of $X$ is a subset of $X$ that is both closed and open in $X$. The (compact Hausdorff) space $X$ is totally disconnected if and only the collection of clopen subsets of $X$ is a base for the topology on $X$. In particular this is the case for the usual Cantor middle-thirds set, which we denote by $\JCantor$.  
We will make tacit use of the standard result that every nonempty, totally disconnected, compact, metrizable space without isolated points is homeomorphic to $\JCantor$, and we call any such space \emph{a Cantor set}.

We denote by $C(X)$ the algebra of all continuous complex-valued functions on $X$ with the supremum norm
$ \|f\|_{X} = \sup\{ |f(x)| : x \in X \}$.
The \emph{real} algebra of all continuous real-valued functions on $X$ we 
denote by $\crx{X}$.

A \emph{uniform algebra} $A$ on $X$ is a closed subalgebra of $C(X)$ that contains the constant functions and separates the points of $X$.
The uniform algebra $A$ on $X$ is said to be \emph{nontrivial} if $A\neq C(X)$.
We denote the maximal ideal space of $A$ by $\ma$. 
The uniform algebra $A$ on $X$ is said to be \emph{natural} if $\ma=X$ (under the usual identification of a point of $X$ with the corresponding complex homomorphism).
We denote by $\inv{A}$ the group of invertible elements of $A$ and by $|A|$ the set of nonnegative real-valued functions $\{|f|: f\in A\} \subset \crx{X}$. For $\sG\subset \inv{A}$ we denote by $\sG^{-1}$ the set $\{g^{-1}:g\in \sG\}$.

Consider a set of functions $B\subset C(X)$, where $B$ need not be an algebra. 
We say that $B$ is \emph{regular on $X$} if for each closed subset $K_0$ of $X$ and each point $x$ of $X\setminus K_0$, there exists a function $f$ in $B$ such that $f=0$ on $K_0$ and $f(x)=1$.  We say that $B$ is \emph{normal on $X$} if for each pair of disjoint closed subsets $K_0$ and $K_1$ of $X$, there exists a function $f$ in $B$ such that $f=0$ on $K_0$ and $f=1$ on $K_1$.
 
Throughout the remainder of this section, $A$ will denote a uniform algebra on $X$.

The uniform algebra $A$ on $X$ is \emph{regular} or \emph{normal} if $A$ is natural and is regular on $X$ or normal on $X$, respectively.
Clearly, if $A$ is normal on $X$, then $|A|$ is normal on $X$, and if $|A|$ is normal on $X$ then $A$ is regular on $X$. Also, if $A$ is regular on $X$ then $X$ is the Shilov boundary for $A$.

It is standard that if $A$ is normal on $X$, then $A$ is natural, and hence, normal.  It is also standard that if $A$ is regular, then $A$ is normal (and the converse holds trivially).  However, the condition that $A$ is regular \emph{on} $X$, does not imply that $A$ is natural, and hence, does not imply that $A$ is regular!  Indeed, as mentioned in the introduction, the Hoffman-Singer fiber algebra $A_\alpha$ is an example of this phenomenon, as are also the examples we construct.  
Note that, as mentioned in the introduction,
a uniform algebra that is regular on its Shilov boundary is regular (equivalently, normal) if and only if its Shilov boundary coincides with its maximal ideal space.

We leave the proof of the following easy lemma to the reader.

\begin{lemma}
\label{equivalent-|A|}
Let $A$ be a uniform algebra on a compact space $X$. Then $|A|$ is normal on $X$ if and only if the following condition holds: for every closed subset $E$ of $X$ and every $y\in X\setminus E$, there exists $g\in A$ such that $g$ vanishes on a neighborhood of $y$ and $|g|=1$ on $E$.
\end{lemma}

The uniform algebra $A$ is \emph{strongly regular at} $x\in X$ if each function in $A$ that vanishes at $x$ is a (uniform) limit of functions in $A$ that vanish on a neighborhood of $x$.  The uniform algebra $A$ is \emph{strongly regular} if $A$ is strongly regular at every point of $X$.  By a result of Wilken~\cite[Corollary~1]{Wilken}, every strongly regular uniform algebra is normal.

The uniform algebra $A$ on $X$ is \emph{logmodular} (on $X$) if the set $\{\log|f|:f \in \inv{A}\}$ is (uniformly) dense in $\crx{X}$, and $A$ is \emph{Dirichlet} (on $X$) if $\{\Re f: f \in A\}$ is dense in $\crx{X}$.  (We denote the real and imaginary parts of a complex-valued function $f$ by $\Re f$ and $\Im f$, respectively.) Every Dirichlet algebra is logmodular, and if $A$ is logmodular on $X$ then the Shilov boundary for $A$ is $X$.

The \emph{essential set} $E$ for a uniform algebra $A$ on a space $X$ is the smallest closed subset of $X$
such that $A$ contains every continuous complex-valued function on $X$ that vanishes on $E$.
For a proof of the existence of the essential set see~\cite[Section~2--8]{Browder}.
Note that $A$ contains every continuous function whose restriction to $E$ lies
in the restriction of $A$ to $E$. The uniform algebra $A$ is said to be \emph{essential} if
$E = X$.

A subset $E$ of $X$ is a \emph{set of antisymmetry} (for the uniform algebra $A$ on $X$) if every function $f$ in $A$ that is real-valued on $E$ is constant on $E$. If $X$ itself is a set of antisymmetry, then $A$ is said to be an \emph{antisymmetric} uniform algebra. It is standard that the maximal sets of antisymmetry are closed in $X$, and that they form a partition of $X$. Moreover, if $Y$ is a maximal set of antisymmetry, then $A|Y$ is an antisymmetric uniform algebra on $Y$. By the Bishop-Stone-Weierstrauss theorem
(see~\cite[Theorem~2.7.5]{Browder} or~\cite[Theorem~12.1]{Stout}), if $A$ is nontrivial, then there must be a maximal set of antisymmetry $Y$ such that $A|Y$ is nontrivial.

Every antisymmetric uniform algebra is essential, but there are many essential uniform algebras that are not antisymmetric. For example, the essential uniform algebras constructed in \cite{FI} are far from antisymmetric.


\section{Ensuring the absence of isolated points}\label{further-prelim}

When proving our main  results, once we have constructed a nontrivial uniform algebra $A$ on a compact space $X$ such that $A$ has the properties we want, we sometimes wish to show that $X$ has no isolated points, or else show that we can modify our example to ensure this. Generally we can prove that our constructions do not introduce isolated points, but we can also eliminate isolated points should they arise. This section is devoted to establishing results that help with both approaches.

Recall that a space $X$ is said to be \emph{scattered} if every nonempty subset of $X$ has an isolated point.
A famous theorem of Walter Rudin \cite{Rudin-scat} asserts that if $X$ is a compact scattered space, then there are no nontrivial uniform algebras on $X$.
Stout \cite[p.~119]{Stout} observed that Rudin's theorem can be obtained as a consequence of properties of the maximal antisymmetric decomposition. However, Stout used the Shilov idempotent theorem in this deduction, and here we wish to observe that this deep theorem is not needed.

The following elementary lemma is surely known, but we sketch a proof for the convenience of the reader.  Here, as usual, we denote the Gelfand transform of a Banach algebra element $f$ by $\widehat f$.

\begin{lemma}
Let $A$ be a uniform algebra on  a compact space $X$, let $S$ be the Shilov boundary for $A$, and suppose that $x \in S$ is an isolated point of $S$. Then $x$ is isolated in the maximal ideal space $\ma$ of $A$, and there exists an idempotent $g \in A$ such that $g(x)=1$ and $\widehat g(y) = 0$ for all $y \in \ma\setminus \{x\}$.
\end{lemma}
\begin{proof}

Let $U$ be a open neighborhood of $x$ in $\ma$ such that $U\cap S = \{x\}$. Since $x$ is in the Shilov boundary for  $A$, there exists $f \in A$ with $\|f\|_X=1$ but $\|\widehat f\|_{\ma\setminus U} < 1.$ Since $\|f\|_S=1$ it follows that $|f(x)|=1$. Multiplying by the constant $\overline{f(x)}$ if necessary, we can assume that $f(x)=1$. Clearly the sequence 
$(f^n)$ converges uniformly on $S$, and hence, converges in $A$ to a function $g\in A$ with $g(x)=1$ and $g(y)=0$ for all $y \in S\setminus \{x\}$.
To see that $\widehat g(y)=0$ for all $y\in\ma\sm\{x\}$, and hence that $x$ is isolated in $\ma$, fix a point $y\in\ma\sm\{x\}$ and consider a function $h\in A$ such that $h(x)=0$ and $\widehat h(y)=1$, and then note that $gh=0$, so $\widehat g(y)=\widehat g(y)\widehat h(y) = \widehat{gh}(y)=0$.
\end{proof}

The following corollaries are now immediate, without the need to appeal to the Shilov idempotent theorem.
\begin{corollary}
Let $A$ be a uniform algebra on a compact space $X$. Suppose that $\ma$ has no isolated points. Then the Shilov boundary for $A$ also has no isolated points. In particular, the conclusion holds whenever $\ma$ is a connected space with more than one point.
\end{corollary}

\begin{corollary}\label{no-idempotents}
Let $A$ be a uniform algebra on a compact space $X$ with more than one point.  Suppose that $A$ has no idempotents other than $0$ and $1$. Then the Shilov boundary for $A$ has no isolated points.
\end{corollary}

\begin{corollary}\label{antisymmetry-perfect}
Let $A$ be an antisymmetric uniform algebra on a compact space $X$ with more than one point. Then the Shilov boundary for $A$ has no isolated points.
\end{corollary}

We now see how Rudin's theorem on scattered spaces follows from Corollary \ref{antisymmetry-perfect}. We include some useful extra information that comes from this approach.  Recall that a subset of a topological space is said to be {\it perfect\/} if it is closed and has no isolated points.

\begin{corollary}\label{Rudin}
Let $A$ be a nontrivial uniform algebra on a compact space $X$. Then $X$ has a nonempty perfect subset $S$ such that $A|S$ is a nontrivial antisymmetric uniform algebra on $S$.
\end{corollary}

\bpf
By the Bishop-Stone-Weierstrauss theorem, there is a maximal set of antisymmetry $Y$ such that the uniform algebra $A|Y$ is nontrivial. Clearly $Y$ has more than one point. Let $S$ be the Shilov boundary for $A|Y$. By Corollary \ref{antisymmetry-perfect}, since $A|Y$ is antisymmetric, $S$ is a nonempty perfect subset of $X$. Moreover, $A|S$ is a nontrivial antisymmetric uniform algebra on $S$.
\epf


\section{The main construction}\label{section-main-construction}

In this section we prove Theorems~\ref{main-theorem} and \ref{main-construction}.

We will need the following elementary lemma.

\blem\label{countable-dense}
Let $X$ be a metric space, let $Y$ be a separable subspace of $X$, and let $Z$ be a dense subspace of $X$.  Then there is a countable subset of $Z$ whose closure in $X$ contains $Y$.
\elem

\bpf
Choose a countable dense subset $S$ of $Y$.  For each $s\in S$ and each $n\in \N$ choose a point $z_{s,n}$ of $Z$ such that $d(s,z_{s,n})< 1/n$ (where $d$ denotes the metric on $X$).  The set $E=\{z_{s,n}: s\in S {\rm\ and\ } n\in \N\}$ is countable and $Y$ is contained in the closure of $E$ in $X$.
\epf

The following lemma is standard.  See for instance \cite[Chapter~II, Section~4]{Hurewicz-Wallman1948} or
\cite[Theorems~1.4.5 and~1.6.5]{Engelking}.

\blem\label{0-d}
Let $X$ be a compact totally disconnected space.  Then for every pair of disjoint closed subsets $\pair K$ of $X$ there is a clopen subset $K$ of $X$ such that $K_1\subset K$ and $K_2\subset X\sm K$.
\elem

The next lemma will be useful also.

\blem\label{countable-pairs}
Let $X$ be a compact metrizable space.  Then there is a countable collection $\sJ$ of pairs $\pair K$ of disjoint closed subsets of $X$ such that for every pair $\pair C$ of disjoint closed subsets of $X$ there is a pair $\pair K\in {\sJ}$ such that $C_1\subset K_1$ and $C_2\subset K_2$.
\elem

\bpf
As a compact metrizable space, $X$ has a countable basis $\sB$.  Let
$${\sC}=\{C: C {\rm\ is\ a\ finite\ union\ of\ closures\ of\ members\ of\ {\sB}}\}.$$
Let ${\sJ}= \{\pair K\in {\sC} \times {\sC} : K_1 {\rm\ and\ } K_2 {\rm\ are\ disjoint}\}$.  The reader can easily verify that $\sJ$ has the required properties.
\epf

\bpf[Proof of Theorem~\ref{main-construction}]
There are three cases to consider: (i) $|A|$ is normal on $X$, (ii) $A$ is logmodular on $X$, (iii) $|A|$ is normal on $X$ and $A$ is logmodular on $X$.  We treat case~(ii) first and then discuss cases (i) and (iii) which follow exactly the same outline.

We will define an increasing sequence $A_0\subset A_1 \subset A_2 \subset \cdots$ of closed unital subalgebras of $A$.  The algebra $B$ will be the closure of
their union in $A$.  We will adopt the following notation.
The canonical quotient map $X\ra \xaj j$ will be denoted by $\pi_j$; the canonical quotient map $X\ra X_B$ will be denoted by $\pi_\infty$.
Given a function $f$ on $X$ that is constant on the fibers of $\pi_j$, and hence induces a function on $\xaj j$, the induced function on $\xaj j$ will be denoted by $f_j$; thus $f_j \circ \pi_j=f$.  For a function $f$ on $X$ that induces a function on $X_B$, the induced function on $X_B$ will be denoted by $f_\infty$. Given a point $p$ in $X$, its equivalence class in $\xaj j$ will be denoted by $p_j$, and its equivalence class in $X_B$ will be denoted by $p_\infty$.
Every function on $\xaj j$ (or $X_B$) gives rise to a function on $X$ by precomposing with $\pi_j$ (or $\pi_\infty$).  Given a set $\sC$ of functions\vadjust{\kern2pt} on $\xaj j$ or $X_B$, we will denote the set of corresponding functions on $X$ by $^X{\sC}$.
Given a set of functions $\sF$ in $A$, we will denote by $[\sF]$ the closed unital subalgebra of $A$ generated by $\sF$.  Recall that we denote by $\inv A$ the set of invertible elements of $A$, and given a set $\sG$ of invertible elements, we denote by $\sG^{-1}$ the set of inverses of the elements of $\sG$.

Choose a function $a\in A$ such that $\ol a\notin A$, and set $A_0=[\{a\}]$.

We will inductively choose two sequences $(\sG_j)_{j=0}^\infty$ and $(\sH_j)_{j=0}^\infty$ of countable subsets of $\inv A$ such that setting
$$\textstyle A_j=\Bigl[\{a\} \cup \bigl(\bigcup_{k=0}^{j-1} \sG_k\bigr) \cup \bigl(\bigcup_{k=0}^{j-1} \sG_k^{-1}\bigr)  \cup \bigl(\bigcup_{k=0}^{j-1} \sH_k\bigr)\Bigr]$$
for $j=1, 2, \ldots$, the following three conditions are satisfied for all $j=0, 1, 2, \ldots$.
\begin{enumerate}
\item[(1)] The closure of the set $\{\log |g|: g\in{\sG_j}\}$ in $\crx X$ contains $^X\crx {\xaj j}$.
\item[(2)] For every function $h\in \sH_j$, the image of the function $|h|$ is contained in $[0, 3/2) \cup (2, \infty)$. 
\item[(3)] For every pair $\pair C$ of disjoint closed subsets of $X_{A_j}$ there is a function $h\in \sH_j$ such that $|h|<3/2$ on $\pi_j^{-1}(C_1)$ and $|h|>2$ on $\pi_j^{-1}(C_2)$.
\end{enumerate}
Let $n$ be a nonnegative integer, and
assume for the purpose of induction that we have chosen $\sG_j$ and $\sH_j$ for all $0\leq j<n$.  (When $n=0$, this assumption is vacuous.)
Note that $A_n$ is separable since it is generated by a countable set of functions.  Therefore, $\xaj n$ is (compact) metrizable, and $^X\crx {\xaj n}$ is separable.
Since $A$ is logmodular, Lemma~\ref{countable-dense} yields a countable set of functions ${\sG}_n$ in $\inv A$ such that the closure of the set $\{\log|g| : g\in {\sG}_n\}$ in $\crx X$ contains $^X\crx {\xaj n}$.
Let $\sJ$ be the countable collection of pairs of disjoint closed subsets of $\xaj n$ given by applying Lemma~\ref{countable-pairs}.  Applying Lemma~\ref{0-d} then yields the existence of a countable collection $\sK$ of clopen subsets of $X$ such that for every pair $\pair K\in \sJ$, there exists $K\in \sK$ such that $\pi_n^{-1}(K_1)\subset K$ and $\pi_n^{-1}(K_2)\subset X\sm K$.  Given $K\in\sK$, the function that is 0 on $K$ and 1 on $X\sm K$  is continuous; since $A$ is logmodular, it follows that there is a function $h_K\in \inv A$ such that $|h_K|< 3/2$ on $K$ and $|h_K|>2$ on $X\sm K$.  We conclude that there is a countable set ${\sH}_n$ of functions in $\inv A$ such that for every function $h\in \sH_n$ 
the image of the function $|h|$ is contained in $[0,3/2) \cup (2,\infty)$ and such that
for every pair $\pair C$ of disjoint closed subsets of $\xaj n$ there is a function $h\in \sH_n$ such that $|h|<3/2$ on $\pi_n^{-1}(C_1)$ and $|h|>2$ on $\pi_n^{-1}(C_2)$.  Thus the desired sequences $(\sG_j)_{j=0}^\infty$ and
$(\sH_j)_{j=0}^\infty$ can be inductively chosen.

Set $$\textstyle B=\Bigl[\{a\} \cup \bigl(\bigcup_{k=0}^\infty \sG_k\bigr) \cup \bigl(\bigcup_{k=0}^\infty \sG_k^{-1}\bigr) \cup \bigl(\bigcup_{k=0}^\infty \sH_k\bigr)\Bigr].$$
Note that $B$ is separable since it is generated by a countable set of functions.

To show that $X_B$ is totally disconnected consider two points $p$ and $q$ of $X$ that represent distinct equivalence classes $p_\infty$ and $q_\infty$ in $X_B$.  Then some function in
$\{a\} \cup \bigl(\bigcup_{k=0}^\infty \sG_k\bigr) \cup \bigl(\bigcup_{k=0}^\infty \sG_k^{-1}\bigr) \cup \bigl(\bigcup_{k=0}^\infty \sH_k\bigr)$ must separate $p$ from $q$, and hence there must be some finite $j$ such that $p_j\neq q_j$.  Then there is a function $h\in \sH_j$ such that $|h(p)|<3/2$ and $|h(q)|>2$.  It follows that $p_\infty$ and $q_\infty$ lie in different components of $X_B$, so $X_B$ is totally disconnected.  

Because the\vadjust{\kern3pt} algebras $^X\crx {\xaj 1}\subset {}^X\crx {\xaj 2}\subset \cdots$ form an increasing sequence, their union\vadjust{\kern3pt}
$\bigcup_{j=0}^\infty {}^X\crx {\xaj j}$ is an algebra.  The (unital) algebra of functions induced by
$\bigcup_{j=0}^\infty {}^X\crx {\xaj j}$ on $X_B$ separates points on $X_B$ and hence by the Stone-Weierstrass theorem is dense in $\crx {X_B}$.  Since the closure of the set
$\{\log|g| : g\in {\sG}_j\}$ contains $^X\crx {\xaj j}$, and every member of $\sG_j$ is invertible in $B$, it follows that $\{\log|g_\infty| : g\in \inv B\}$ is dense in $\crx {X_B}$.  Thus $\tB$ is logmodular.

Finally, note that $\tB$ is nontrivial because the function $a_\infty$ is in $\tB$ while the function $\ol a_\infty$ is not. 
This completes the proof of case~(ii).

The proofs of cases (i) and (iii) follow exactly the same pattern.  Consider case~(i).  We again begin by choosing a function $a\in A$ such that $\ol a\notin A$, and set $A_0=[\{a\}]$.
We then inductively choose a single sequence $(\sF_j)_{j=0}^\infty$ of countable subsets of $A$ such that setting
$$\textstyle A_j=\Bigl[\{a\} \cup \bigl(\bigcup_{k=0}^{j-1} \sF_k\bigr)\Bigr]$$
for $j=1, 2, \ldots$, we have all $j=0, 1, 2, \ldots$ that
\begin{enumerate}
\item[(4)] For every function $f\in \sF_j$, the image of the function $|f|$ is contained in 
$\{0,1\}$. 
\item[(5)] For every pair $\pair C$ of disjoint closed subsets of $X_{A_j}$ there is a function $f\in \sF_j$ such that $f=0$ on $\pi_j^{-1}(C_1)$ and $|f|=1$ on $\pi_j^{-1}(C_2)$.
\end{enumerate}
Assume for the purpose of induction that we have chosen $\sF_j$ for all $0\leq j<n$.
Let $\sJ$ and $\sK$ be obtained in the same manner as in the proof of case~(ii).  Since $|A|$ is normal on $X$, given $K\in\sK$, there is a function $f_K\in A$ such that $f_K=0$ on $K$ and $|f_K|=1$ on $X\sm K$.
Consequently, there is a countable set $\sF_n$ of functions in $A$ such that for every function $f\in \sF_n$ the image of the function $|f|$ is contained in $\{0,1\}$ and such that for every pair
$\pair C$ of disjoint closed subsets of $\xaj n$ there is a function $f\in \sF_n$ such that $f=0$ on $\pi_n^{-1}(C_1)$ and $|f|=1$ on $\pi_n^{-1}(C_2)$.
Thus the desired sequence $(\sF_j)_{j=0}^\infty$ can be inductively chosen.

Set $$\textstyle B=\Bigl[\{a\} \cup \bigl(\bigcup_{k=0}^\infty \sF_k\bigr)\Bigr].$$

Consider an arbitrary pair $\pair L$ of disjoint closed subsets in $X_B$.  For each pair $(p,q)
\in \pi_\infty^{-1}(L_1)\times \pi_\infty^{-1}(L_2)$ there is a function in $\{a\} \cup \bigl(\bigcup_{k=0}^\infty \sF_k\bigr)$ that separates $p$ and $q$.  Consequently, a compactness argument shows that the same holds with $\{a\} \cup \bigl(\bigcup_{k=0}^\infty \sF_k\bigr)$ replaced by some finite subset of itself.  Thus there is some finite $j$ such that the closed sets $\pi_j\bigl(\pi_\infty^{-1}(L_1)\bigr)$ and $\pi_j\bigl(\pi_\infty^{-1}(L_2)\bigr)$ of $\xaj j$ are disjoint.  Then by our choice of $\sF_j$, there is a function $f\in \sF_j$ such that $f=0$ on $\pi_\infty^{-1}(L_1)$ and $|f|=1$ on $\pi_\infty^{-1}(L_2)$.  The function $f_\infty$ is in $\tB$ and satisfies $f_\infty=0$ on $L_1$ and $|f_\infty|=1$ on $L_2$.
Thus $|\tB|$ is normal on $X_B$.  Since $|f_\infty|$ takes the values 0 and 1 only, this argument also shows that $X_B$ is totally disconnected.  As before $\tB$ is nontrivial because $a_\infty$ is in $\tB$ while $\ol a_\infty$ is not. 

For the proof of case~(iii), after setting $A_0=[\{a\}]$, we inductively choose two sequences $(\sG_j)_{j=0}^\infty$ and $(\sF_j)_{j=0}^\infty$ as above such that conditions (1), (4), and (5) above are satisfied for all $j=0, 1, 2,\ldots$ with
$$\textstyle A_j=\Bigl[\{a\} \cup \bigl(\bigcup_{k=0}^{j-1} \sG_k\bigr) \cup \bigl(\bigcup_{k=0}^{j-1} \sG_k^{-1}\bigr) \cup \bigl(\bigcup_{k=0}^{j-1} \sF_k\bigr)\Bigr]$$
for $j=1, 2, \ldots$.  (The sequence $(\sH_j)_{j=0}^\infty$ is not needed in case~(iii).)
We then set
$$\textstyle B=\Bigl[\{a\} \cup \bigl(\bigcup_{k=0}^\infty \sG_k\bigr) \cup \bigl(\bigcup_{k=0}^\infty \sG_k^{-1}\bigr) \cup \bigl(\bigcup_{k=0}^\infty \sF_k\bigr)\Bigr].$$
The proof that $\tB$ is nontrivial and logmodular on $X_B$, that $|\tB|$ is normal on $X_B$, and that $X_B$ is totally disconnected is a repetition of arguments above.
\epf

Note that in all three cases of Theorem~\ref{main-construction}, $X_B$ is the Shilov boundary for $\tB$. If $A$ has no nontrivial idempotents, then neither has $\tB$, and so $X_B$ has no isolated points by Corollary~\ref{no-idempotents}.

\bpf[Proof of Theorem~\ref{main-theorem}]
As usual we denote by $H^\infty(D)$ the algebra of all bounded holomorphic functions on the open unit disc $D$.
Let $X$ be the intersection of the fiber of the maximal ideal space of $H^\infty(D)$ over the point $1$ with the Shilov boundary $S$ for $H^\infty(D)$.  
It is well known that $S$, and hence $X$, is totally disconnected \cite[pp.~169-175]{Hoffman}.
Let $A$ be the algebra obtained by restricting the Gelfand transforms of the functions in $H^\infty(D)$ to $X$.  Then 
as shown in \cite[pp.~189--190]{Hoffman} (see also \cite{Hoffman-Singer}),
$A$ is a uniform algebra on $X$ that has no nontrivial idempotents, $X$ is the Shilov boundary for $A$, and $|A|$ is normal on $X$.  (The algebra $A$ is the Hoffman-Singer fiber algebra $A_\alpha$ with $\alpha=1$.)  It is well known that $H^\infty(D)$ is a logmodular algebra \cite[p.~182]{Hoffman}, and it follows that $A$ is logmodular on $X$. Let $\tB$ be the uniform algebra on $X_B$ obtained by applying Theorem~\ref{main-construction} to $A$. Then $X_B$ is compact, totally disconnected, and metrizable. Since $X_B$ is the Shilov boundary for $\tB$, as noted above, by Corollary~\ref{no-idempotents}, $X_B$ has no isolated points. Thus $X_B$ is a Cantor set, and $\tB$ is a uniform algebra with all the required properties.
\epf

It is standard that there is a function in $H^\infty(D)$ whose real part extends continuously to the boundary of $D$ but whose imaginary part does not~\cite[p.~377]{Garnett}.  It follows that the Hoffman-Singer fiber algebra $A_1$ is not antisymmetric.  Consequently, the algebra constructed in the proof of Theorem~\ref{main-theorem} can also fail to be antisymmetric (if, for instance, one chooses the 
sequence $(\sG_j)_{j=0}^\infty$ in the proof of Theorem~\ref{main-construction} to include a nonconstant real-valued function).  However, by applying Corollary~\ref{Rudin}, we can obtain an antisymmetric uniform algebra with the properties in Theorem~\ref{main-theorem}.  Note also that invoking Corollary~\ref{Rudin} eliminates the need to prove that the space $X_B$ in the proof above has no isolated points. 


\section{Uniform algebras logmodular and regular on general spaces}\label{general-spaces}

In this section we first prove Theorem \ref{examples-arbitrary-spaces} and Corollary~\ref{unit-interval}. We then prove, under certain conditions, some results concerning properties preserved by the construction introduced in \cite{FI}. Finally we combine our results to prove Theorem~\ref{essential-arbitrary-spaces}.

A compact space $X$ is said to be {\em simply coconnected\/} if the first \v Cech cohomology group $\check H^1(X;\Z)$ vanishes.  It is well known that simple coconnectivity of $X$ is equivalent to the statement that every zero-free continuous complex-valued function on $X$ has a continuous logarithm.  (See for instance \cite[Theorem~15.4]{Alexander-Wermer}.)  Simple coconnectivity will play a prominent role in the proofs in this section.  It is standard that every totally disconnected compact space, and in particular every closed subset of a Cantor set, is simply coconnected.  Also it is easily shown that every compact subset of $\R$ is simply coconnected.  Although the unit circle $\T$ is not simply coconnected, every \emph{proper} closed subset of $\T$ is simply coconnected.

\bpf[Proof of Theorem~\ref{examples-arbitrary-spaces}]
Let $K$ be a Cantor set contained in $X$.  By Theorem~\ref{main-theorem}, there is a nontrivial uniform algebra $B$ such that $B$ is logmodular on $K$ and $|B|$ is normal on $K$.  Define $A$ by
$$A=\{f\in C(X): f|K\in B\}.$$
Then $A$ is a nontrivial uniform algebra on $X$.

We first show that $|A|$ is normal on $X$.  Consider disjoint closed subsets $K_0$ and $K_1$ in $X$.  By the normality of $|B|$, there is a function $b\in B$ such that $b=0$ on $K_0\cap K$ and $|b|=1$ on $K_1\cap K$.  As a closed subset of a Cantor set, $K_1\cap K$ is simply coconnected.  Therefore, there is a continuous real-valued function $\sigma: K_1\cap K\ra\R$ such that $b=e^{i\sigma}$ on $K_1\cap K$.  By the Tietze extension theorem, $\sigma$ extends to a continuous function $\tsigma:K_1\ra\R$.
There is a well-defined continuous function $a:K_0\cup K_1\cup K\ra \C$ given $$
a(x) =
\begin{cases} 0 &\mbox{for\ } x\in K_0\\
e^{i\tsigma(x)} & \mbox{for\ } x\in K_1\\
b(x) & \mbox{for\ } x\in K.
\end{cases}
$$
Apply the Tietze extension theorem again to extend this function to a continuous function $\ta$ on $X$.  Then $\ta=0$ on $K_0$ and $|\ta|=1$ on $K_1$; and $\ta$ is in $A$ since $\ta|K=b$.  Thus $|A|$ is normal on $X$.

We now show that $A$ is logmodular on $X$.
Let $u\in \crx X$ be arbitrary and fix $\vep>0$.  Since $B$ is logmodular on $K$, there is a function $f\in \inv B$ such that $\bigr\|\,u|K - \log |f|\, \bigl\|_K < \vep$.  Since $K$ is simply coconnected, there is a function $g\in C(K)$ such that $f=e^g$.  (Note that it is not claimed that $g$ is in $B$.)  Then $\Re g = \log |f|$, so $\| u|K -\Re g \|_K < \vep$.  Set 
$\alpha= u|K - \Re g$, and apply the Tietze extension theorem to obtain a continuous real-valued function $\talpha$ on $X$ satisfying $\|\talpha\|_X=\|\alpha\|_K<\vep$.  Set $w=u-\talpha$.  Then $w$ is a continuous real-valued function on $X$ and $w|K=u|K - \alpha =\Re g$.  Extend $\Im g$ to a real-valued continuous function $v$ on $X$.  Then set $h= w+i v$.  Then on $K$ we have $e^h=e^{w+iv} = e^g=f$.  Thus $e^h$ is in $A$.  Furthermore, on $K$ we also have $e^{-h}=f^{-1}\in B$.  Thus $e^{-h}$ is in $A$, so $e^h$ is invertible in $A$.  Finally note that
$$\|\, u - \log|e^h| \, \|_X = \|u -\log e^w\|_X = \|u-w\|_X = \|\talpha\|_X < \vep.$$

The uniform algebra $A$ is not normal since its restriction to the Cantor set $K$ is the nontrivial uniform algebra $B$.
\epf

\bpf[Proof of Corollary~\ref{unit-interval}]
This is immediate from Theorem~\ref{examples-arbitrary-spaces}, since the closed unit interval contains the Cantor set $\JCantor$.
\epf

To prove Theorem~\ref{essential-arbitrary-spaces} we will use the general method for constructing essential uniform algebras given in \cite{FI}.  We first recall the relevant parts of the main theorem of \cite{FI}.

\bthm\label{general-method}
Let $A$ be a nontrivial uniform algebra on a compact space $K$.  Let $X$ be a compact metric space every
nonempty open subset of which contains a nowhere dense subspace homeomorphic to $K$.
Then there exists a sequence $\{K_n\}_{n=1}^\infty$ of pairwise disjoint, nowhere dense subspaces of $X$ each homeomorphic to $K$ such that $\bigcup_{n=1}^\infty K_n$ is dense in $X$ and ${\rm diam} (K_n)\rightarrow 0$.  If  homeomorphisms $h_n:K_n\rightarrow K$ are chosen and we set $A_n=\{f\circ h_n:f\in A\}$, then the set of functions $A_X=\{f\in C(X): f|K_n\in A_n \hbox{\ for\ all\ $n$}\}$ is an essential uniform algebra on $X$.
The equality $A_X|K_n=A_n$ holds for all $n$.
Furthermore, $A_X$ is regular if and only if $A$ is regular.
\ethm

Several additional relations between the properties of $A$ and the properties of $A_X$ are proven in \cite{FI} but need not concern us here.  We will, however, need two additional relations not considered in \cite{FI}.

\bthm\label{essential-pseudo-normal}
In the context of Theorem~\ref{general-method}, if $|A_X|$ is normal on $X$ then $|A|$ is normal on $K$. The converse holds provided that every proper closed subset of $K$ is simply coconnected.
\ethm

\bpf
The proof that 
$|A|$ is normal on $K$ whenever $|A_X|$ is normal on $X$ is essentially trivial and hence left to the reader.

We now suppose that every proper closed subset of $K$ is simply coconnected and that $|A|$ is normal on $K$.
Let $E$ be a proper closed subset of $X$, and let $x \in X\setminus E$. We will show that there exists $f\in A_X$ such that $|f|=1$ on $E$ and $f=0$ on a neighborhood of $x$ in $X$, and hence $|A_X|$ is normal on $X$, by Lemma~\ref{equivalent-|A|}. We consider two cases.

First suppose that $x$ lies in none of the sets $K_n$ that intersect $E$.  (The point $x$ may or may not lie in some $K_n$ that is disjoint from $E$.)  
In this case, let $Z$ denote the quotient space obtained by identifying each $K_n$ for $n\in \N$ to a point, and let $p:X\ra Z$ be the canonical quotient map.  The space $Z$ is compact and metrizable~\cite[Lemma~2.6]{FI}. Then $p(x) \notin p(E)$, so Urysohn's lemma provides a function $f^* \in \crx Z$ such that $f^*=1$ on $p(E)$ and $f^*=0$ on a neighborhood of $p(x)$. The function $f=f^*\circ p$ is in $\crx X$ and $f$ is constant on each $K_n$, so $f\in A_X$. Clearly $f$ has the required properties.

Now suppose instead that $x$ lies in $K_{n_0}$ for some $n_0 \in \N$ such that $K_{n_0}$ intersects $E$.
Since $|A_{n_0}|$ is normal on $K_{n_0}$, there exists $g \in A_{n_0}$ such that $|g|=1$ on $E\cap K_{n_0}$ and $g=0$ on a relatively open neighborhood $U$ of $x$ in $K_{n_0}$. Since $E \cap K_{n_0}$ is simply coconnected, there is a continuous real-valued function $\sigma:E \cap K_{n_0}\to \R$ such that $g|E \cap K_{n_0} = e^{i\sigma}$.

Now let $Z$ denote the quotient space (again compact and metrizable) obtained by identifying the set $K_n$ to a point for each $n\in \N\setminus\{n_0\}$, and let $p:X\ra Z$ be the canonical quotient map.  Let $g^*:p(K_{n_0})\to \C$ be the function in $C(p(K_{n_0}))$ such that $g^*\circ (p|K_{n_0}) = g$, and let $\sigma^*$ be the function in $\crx {p(E \cap K_{n_0})}$ such that $\sigma^* \circ (p|E \cap K_{n_0}) = \sigma$. Then $g^*|p(E \cap K_{n_0})= e^{i \sigma^*}$. By the Tietze extension theorem we can extend $\sigma^*$ to a function $\tss$ in $\crx{p(E)}$.

We have that $p(x) \notin p(E)$. Choose a closed neighborhood $L^*$ of $p(x)$ in $Z$
such that $L^*$ and $p(E)$ are disjoint and $L^*\cap p(K_{n_0}) \subset p(U)$. 
There is a well-defined continuous function $f^*:p(E)\cup p(K_{n_0}) \cup L^*\ra \C$ given by
$$
f^*(x) =
\begin{cases}
0 & \mbox{for\ } x\in L^*\\
e^{i\tss(x)} &\mbox{for\ } x\in p(E)\\
g^*(x)& \mbox{for\ } x\in p(K_{n_0}).
\end{cases}
$$
By the Tietze extension theorem, we can extend $f^*$ to a function $\tfs$ in $C(Z)$. The function $f=\tfs\circ p$ is in $C(X)$, $f|K_{n_0}=g \in A_{n_0}$, and $f|K_n$ is constant for all $n \in \N\setminus\{n_0\}$. Thus $f\in A_X$, and $f$ has the desired properties.
\epf

\bthm\label{essential-logmodular}
In the context of Theorem~\ref{general-method}, if $A_X$ is logmodular on $X$ then $A$ is logmodular on $K$.  The converse holds provided $K$ is simply coconnected.
\ethm

The proof of this theorem will use the following elementary lemma whose proof is given in \cite[Lemma~6.3]{Izzo}.

\blem\label{S-W}
In the context of Theorem~\ref{general-method}, the set
$$G=\{ f\in \crx X : f|K_n {\rm \ is\  constant\ for\ all\ but\ finitely\ many\ } n\}$$
is dense in $\crx X$.
\elem

\bpf[Proof of Theorem~\ref{essential-logmodular}]
The proof is similar to the proof that $A_X$ is Dirichlet on $X$ if and only if $A$ is Dirichlet on $K$ given in \cite[Theorem~6.2]{Izzo}.

That $A$ is logmodular on $K$ whenever $A_X$ is logmodular on $X$ is essentially trivial and hence left to the reader.

We now suppose that $K$ is simply coconnected and that $A$ is logmodular on $K$ and will prove that $A_X$ is logmodular on $X$.  Let $u$ in $\crx X$ and $\vep>0$ be arbitrary.  By Lemma~\ref{S-W}, there is a positive integer $N$ and a function $w$ in $\crx X$ such that $w|K_n$ is constant for all $n>N$ and $\|u-w\|_X< \vep/2$.  Since $A$ is logmodular on $K$, there is, for each $n=1,\ldots, N$, a function $\alpha_n$ in $\crx {K_n}$ with $\|\alpha_n\|_{K_n}<\vep/2$ such that $w|K_n - \alpha_n=\log |a_n|$ for some $a_n$ in $\inv A_n$.  Since $K$ is simply coconnected, there is a function $b_n$ in $C(K_n)$ such that $a_n=e^{b_n}$.

Let $Z$ denote the quotient space obtained by identifying each $K_n$ for $n>N$ to a point, and let $p:X\ra Z$ be the canonical quotient map.  The compact space $Z$ is metrizable \cite[Lemma~2.6]{FI}.

Let $w^*$ be the function in $\crx Z$ such that $w=w^*\circ p$.  For each $n=1,\ldots, N$, let $b_n^*$ and $\alpha_n^*$ be the functions in $C\bigl(p(K_n)\bigr)$ and $C_\R\bigl(p(K_n)\bigr)$, respectively, such that $b_n=b_n^*\circ(p|K_n)$ and $\alpha_n=\alpha_n^*\circ (p|K_n)$.  By the Tietze extension theorem, there is a function $\alpha^*$ in $\crx Z$ with $\|\alpha^*\|_Z<\vep/2$ such that $\alpha^*|\bigl(p(K_n)\bigr)=\alpha_n^*$ for every $n=1,\ldots, N$.  The Tietze extension theorem also yields a function $s$ in $\crx Z$ such that $s|\bigl(p(K_n)\bigr) = \Im b_n^*$ for every $n=1,\ldots, N$. 

Set $h=\exp\bigl[ (w^* - \alpha^*) + is\bigr] \circ p$.  On each $K_n$, for $n=1,\ldots, N$, the function $h$ coincides with $a_n$.  On each $K_n$, for $n>N$, of course $h$ is constant.  Thus $h$ belongs to $A_X$.  Furthermore,
\begin{align*}
\bigl\| u - \log |h| \, \bigr\|_X
&\leq \| u - w \|_X + \bigl\| w - \log |h| \, \bigr\|_X \\
&=\| u - w \|_X + \bigl\|w - \bigl((w^*\circ p) - \alpha^*\circ p\bigr) \bigr\|_X \\
&= \|u - w\|_X + \|\alpha^*\|_Z  < \vep.
\end{align*}
\epf

Before finally proving Theorem~\ref{essential-arbitrary-spaces}, we present two elementary lemmas.

\blem\label{claim-4}
Every Cantor set $K$ contains a Cantor set $C$ that is nowhere dense in $K$.
\elem

\bpf
Without loss of generality $K=\{0,1\}^\omega$.  The set $C=\{(x_n)\in \{0,1\}^\omega : x_n=0 {\rm\ for\ all\ } n {\rm\ odd}\}$ has the required properties.
\epf

\blem\label{claim-5}
Let $X$ be a compact metrizable space without isolated points.  Then every nonempty open subset of $X$ contains a nowhere dense Cantor set.
\elem

\bpf
Let $U$ be a nonempty open subset of $X$.  Then $U$ contains a nonempty open set $V$ with $\ol V\subset U$. Then $\ol V$ is a nonempty perfect compact subset of the metrizable space $X$ and hence contains a Cantor set.  The desired nowhere dense Cantor set is now obtained by applying the preceding lemma.
\epf

\bpf[Proof of Theorem~\ref{essential-arbitrary-spaces}]
By Lemma~\ref{claim-5}, every open subset of $X$ contains a nowhere dense Cantor set.  Therefore,
we can apply Theorem~\ref{general-method} to the uniform algebra on the Cantor set given by Theorem~\ref{main-theorem} to obtain a uniform algebra on $X$ that has the desired properties by Theorems~\ref{essential-pseudo-normal} and~\ref{essential-logmodular}.
\epf


\section{Dense Squares on the Cantor set}\label{dense-squares}

In this section we prove Theorem~\ref{optional-theorem} by applying Cole's method of root extensions to the uniform algebra of Theorem~\ref{main-theorem}.  First we recall some standard results related to the Cole construction.  Then we show that, under suitable conditions, some of the properties we are interested in are preserved by this construction. Finally we use these results to prove Theorem~\ref{optional-theorem}.

Systems of Cole extensions have been discussed in detail in many of our earlier papers.
The following proposition includes the facts that we will need.
\begin{proposition}
\label{metric-Cole}
Let $A$ be uniform algebra on a compact metrizable space $X$. Then there exist a uniform algebra $\jt{A}$ on a compact metrizable space $\jt{X}$, a continuous surjective map $\pi:\jt{X} \to X$, and a norm one continuous linear operator $T:C(\jt{X}) \to C(X)$ with the following properties:
\begin{enumerate}
\item[(a)] for all $f\in C(X)$, $T(f\circ \pi)=f$;
\item[(b)] $\{f\circ \pi: f \in A\} \subset \jt{A}$ and $T(\jt{A})=A$;
\item[(c)] $\jt{A}$ is nontrivial if and only if $A$ is nontrivial;
\item[(d)] $\jt{A}$ is regular on $\jt{X}$ if and only if $A$ is regular on $X$;
\item[(e)] $\jt{A}$ is natural on $\jt{X}$ if and only if $A$ is natural on $X$;
\item[(f)] $\{f^2:f \in \jt{A}\}$ is a dense subset of $\jt{A}$;
\item[(g)] for each $x\in X$, the fiber $\pi^{-1}(x)$ is totally disconnected, and $\jt{A}|\pi^{-1}(x)$ is dense in $C(\pi^{-1}(x)).$
\item[(h)] for each $x\in X$, the fiber $\pi^{-1}(x)$ is a Cantor set.
\end{enumerate}
\end{proposition}

Everything in the above proposition is standard (see \cite{Cole}, \cite{F1}, and \cite{Stout}) except for property~(h).  Although we will use this property in our proof of Theorem~\ref{optional-theorem}, we will also point out how its use can be avoided. Property~(h) can be easily verified by anyone familiar with the Cole construction.
One observes that each time one adjoins square roots to a countable dense set of functions, each fiber is a countable product of discrete two-point spaces and hence is a Cantor set; then one notes that one also obtains Cantor set fibers when one takes the inverse limit of the resulting sequence of spaces to obtain the final space $\jt{X}$.

In the next lemma, whose easy proof we leave to the reader, we denote the characteristic function of a set $E$ by $\chi_E$.

\begin{lemma}
Let $A$ be a uniform algebra on a compact space $X$, and suppose that $|A|$ is normal on $X$.
Suppose further that $X$ can be partitioned into finitely many pairwise disjoint clopen subsets $X_1,X_2,\dots,X_n$. Then, for all nonnegative real numbers $\alpha_1,\alpha_2,\dots,\alpha_n$, the function $\sum_{k=1}^n \alpha_k \chi_{X_k}$ is in $|A|$.
\end{lemma}

We now extend Proposition \ref{metric-Cole} by showing that certain additional properties are preserved when $X$ is totally disconnected. In the proof we denote by $D(z,r)$ the open disc in $\C$ with center $z$ and radius $r$.

\begin{theorem}\label{extended-Cole}
Let $A$, $X$, $\jt{A}$, $\jt{X}$, $\pi$ and $T$ be as in Proposition \ref{metric-Cole}. Suppose that $X$ is totally disconnected and that $|A|$ is normal on $X$. Then $\jt X$ is totally disconnected, and $|\jt A|$ is normal on $\jt{X}$.
\end{theorem}
\begin{proof}
Since $X$ is totally disconnected, the fact that $\jt{X}$ is also totally disconnected follows easily from Proposition \ref{metric-Cole}(g).

We prove that $|\jt A|$ is normal on $\jt{X}$ by verifying the condition in Lemma \ref{equivalent-|A|}.

Let $E$ be a closed subset of $\jt{X}$ and let $y \in \jt{X} \setminus E$. Set $x=\pi(y)$.

Suppose first that $x \notin \pi(E)$. Then there exists $f\in A$ such that $f$ vanishes on a neighborhood of $x$ and $|f|=1$ on $\pi(E)$. Set $g=f\circ \pi$. Then $g$ is in $\jt{A}$ by Proposition~\ref{metric-Cole}(b) and vanishes on a neighborhood of $y$, and $|g|=1$ on $E$, as required.

Now suppose instead that $x \in \pi(E)$, and set $Y=\pi^{-1}(x)$. Since $Y$ is totally disconnected, we can choose a relatively clopen subset $K$ of $Y$ that contains $y$ and is disjoint from $E$.  Since $\jt{A}|Y$ is dense in $C(Y)$, there exists $h \in \jt{A}$ such that $\|h\|_K < 1/3$ and $\|h-1\|_{Y\sm K}<1/3$.
Now choose a compact neighborhood $N$ of $x$ such that $h(\pi^{-1}(N)) \subset D(0,1/3) \cup D(1,1/3)$.
Shrinking $N$ if necessary, we can also assume that $\|h-1\|_{E\cap \pi^{-1}(N)} <1/3$.
Set $\jt{N}=\pi^{-1}(N)$.

By Runge's theorem, there exists a sequence of polynomials $(p_n)$ converging uniformly to $0$ on $D(0,1/3)$ and uniformly to $1$ on $D(1,1/3)$. Then the sequence $(p_n\circ h)$ converges uniformly on $\jt{N}$ to an idempotent in $C(\jt{N})$ that is identically $0$ on a neighborhood of $y$ and identically $1$ on $E\cap \jt{N}$.

Since $|A|$ is normal on  $X$, there exists $f \in A$ such that $f=0$ on $X\setminus N$ and $|f|=1$ on some clopen neighborhood $V$ of $x$. Then the sequence of functions $\bigl((f\circ \pi) (p_n\circ h)\bigr)$ converges uniformly on $\jt{X}$ to a function $a \in \jt{A}$ that vanishes on $\jt{X} \setminus \jt{N}$ and on a neighborhood of $y$ and has modulus $1$ on $E \cap \pi^{-1}(V)$.

Since $|A|$ is normal on  $X$, there exist functions $f_1$ and $f_2$ in $A$ such that $|f_1|=\chi_V$ and $|f_2|=\chi_{X\setminus V}$ on $X$. Set $g= a (f_1\circ \pi) + f_2 \circ \pi \in \jt{A}$. Then $g$ vanishes on a neighborhood of $y$ and has modulus $1$ on $E$, as required.
\end{proof}

Given the above results, the proof of Theorem~\ref{optional-theorem} is quite short.

\bpf[Proof of Theorem~\ref{optional-theorem}]
By Theorem~\ref{main-theorem}, there is a nontrivial uniform algebra $A$ on the Cantor set $\JCantor$ such that $|A|$ is normal on $\JCantor$. We apply Proposition~\ref{metric-Cole} to $A$, with $X=\JCantor$, to obtain a nontrivial uniform algebra $\jt{A}$ on the compact metrizable space $\jt{X}$. 
By Proposition~\ref{metric-Cole}(f), $\{f^2:f\in\jt{A}\}$ is dense in $\jt{A}$.
By Theorem~\ref{extended-Cole}, $|\jt{A}|$ is normal on $\jt{X}$. 
Also by Theorem~\ref{extended-Cole}, $\jt{X}$ is totally disconnected, and by Proposition~\ref{metric-Cole}(h), $\jt{X}$ has no isolated points.  Thus $\jt{X}$ is a Cantor set.
\epf

The use of Proposition~\ref{metric-Cole}(h) in the above proof can be avoided by invoking Corollary~\ref{Rudin} to eliminate isolated points (that are actually not present).  With that approach one obtains a uniform algebra with the additional property of being antisymmetric.


\end{document}